\documentclass[12pt, a4paper]{amsart}
\usepackage{verbatim}
\usepackage[dvips]{color}
\usepackage{amssymb}
\usepackage[active]{srcltx}
\usepackage{latexsym}
\usepackage{amsmath}
 \usepackage{float}
\usepackage{mathrsfs} %fonts veramente calligrafici con \mathscr{B}
\usepackage[draft]{fixme}
\usepackage{graphicx, epstopdf, float}

\allowdisplaybreaks

\newtheorem{theorem}{Theorem}[section]
\newtheorem{proposition}{Proposition}[section]
\newtheorem{lemma}{Lemma}[section]
\newtheorem{corollary}{Corollary}[section]
\theoremstyle{definition}

\newtheorem{remark}{Remark}[section]

\newtheorem{assumption}{Assumption}

\floatstyle{ruled} \restylefloat{figure} \restylefloat{table}

\numberwithin{equation}{section}

\newcommand{\beq}{\begin{equation}}
\newcommand{\bea}[1]{\begin{array}{#1} }
\newcommand{\eeq}{ \end{equation}}
\newcommand{\ea}{ \end{array}}

\def\mean#1{\mathchoice%
          {\mathop{\kern 0.2em\vrule width 0.6em height 0.69678ex depth -0.58065ex
                  \kern -0.8em \intop}\nolimits_{\kern -0.4em#1}}%
          {\mathop{\kern 0.1em\vrule width 0.5em height 0.69678ex depth -0.60387ex
                  \kern -0.6em \intop}\nolimits_{#1}}%
          {\mathop{\kern 0.1em\vrule width 0.5em height 0.69678ex
              depth -0.60387ex
                  \kern -0.6em \intop}\nolimits_{#1}}%
          {\mathop{\kern 0.1em\vrule width 0.5em height 0.69678ex depth -0.60387ex
                  \kern -0.6em \intop}\nolimits_{#1}}}

\def\vintslides_#1{\mathchoice%
          {\mathop{\kern 0.1em\vrule width 0.5em height 0.697ex depth -0.581ex
                  \kern -0.6em \intop}\nolimits_{\kern -0.4em#1}}%
          {\mathop{\kern 0.1em\vrule width 0.3em height 0.697ex depth -0.604ex
                  \kern -0.4em \intop}\nolimits_{#1}}%
          {\mathop{\kern 0.1em\vrule width 0.3em height 0.697ex depth -0.604ex
                  \kern -0.4em \intop}\nolimits_{#1}}%
          {\mathop{\kern 0.1em\vrule width 0.3em height 0.697ex depth -0.604ex
                  \kern -0.4em \intop}\nolimits_{#1}}}

\newcommand{\aveint}[2]{\mathchoice%
          {\mathop{\kern 0.2em\vrule width 0.6em height 0.69678ex depth -0.58065ex
                  \kern -0.8em \intop}\nolimits_{\kern -0.45em#1}^{#2}}%
          {\mathop{\kern 0.1em\vrule width 0.5em height 0.69678ex depth -0.60387ex
                  \kern -0.6em \intop}\nolimits_{#1}^{#2}}%
          {\mathop{\kern 0.1em\vrule width 0.5em height 0.69678ex depth -0.60387ex
                  \kern -0.6em \intop}\nolimits_{#1}^{#2}}%
          {\mathop{\kern 0.1em\vrule width 0.5em height 0.69678ex depth -0.60387ex
                  \kern -0.6em \intop}\nolimits_{#1}^{#2}}}

\def\eqn#1$$#2$${\begin{equation} \label#1#2\end{equation}}
\def\charfn_#1{{\raise1.2pt\hbox{$\chi
_{\kern-1pt\lower3pt\hbox{{$\scriptstyle#1$}}}$}}}

\def\qq1{q_*}
\def\q2{q_{**}}

\newdimen\vintbar
\def\vint{-\kern-\vintbar\int}

\def\0{\boldsymbol 0}

\newtoks\by
\newtoks\paper
\newtoks\book
\newtoks\jour
\newtoks\yr
\newtoks\pages
\newtoks\vol
\newtoks\publ

\def\name[#1, #2]{#1 #2}
\def\ota{{\hbox{\bf ???}}}
\def\cLear{\by = \ota\paper = \ota\book = \ota\jour = \ota\yr = \ota
\pages = \ota\vol = \ota\publ = \ota}
\def\endpaper{\the\by, \textit{\the\paper},
{\the\jour} \textbf{\the\vol} (\the\yr), \the\pages.\cLear}
\def\endbook{\the\by, \textit{\the\book},
\the\publ, \the\yr.\cLear}
\def\endpap{\the\by, \textit{\the\paper}, \the\jour.\cLear}
\def\endproc{\the\by, \textit{\the\paper}, \the\book, \the\publ,
\the\yr, \the\pages.\cLear}

\begin{document}

\title[From Moment Explosion to the Distribution Tails]{From Moment Explosion to the Asymptotic Behavior of the Cumulative Distribution for a Random Variable}

\author{ Sidi Mohamed Aly}

\address{Sidi Mohamed Aly\\Department of Mathematics, Uppsala University\\
S-751 06 Uppsala, Sweden}
\email{souldaly@math.uu.se \\ souldaly@gmail.com}

\maketitle

\begin{abstract}
\noindent We study the Tauberian relations between the moment generating function (MGF) and the complementary cumulative distribution function of a random variable whose MGF is finite only on part of the real line. We relate the right tail behavior of the cumulative distribution function of such a random variable to the behavior of its MGF near the critical moment. We apply our results to an arbitrary superposition of a CIR process and the time-integral of this process.

\medskip

\noindent

\medskip

\end{abstract}

\section{Introduction}
\noindent Tauberian theory is a powerful tool to deal with the problem of obtaining asymptotic information about a function  from a priori knowledge of the asymptotic behavior of some integral transform of the function. The study of Tauberian type theorems has been historically stimulated by their potential applications in diverse fields of mathematics. In mathematical finance, they allow to derive closed form approximate solutions for large values asymptotic analysis.  

The present paper investigates the Tauberian relations between the moment generating function of a random variable and its complementary cumulative distribution. Intuitively, if we assume that we have a positive random variable $Z$ whose density is given by $P_Z(x) = \frac{1}{\nu}Êe^{- \nu x} 1_{x \geq 0} $, then it is easy to see that MGF of $Z$ is only finite on $[0, \nu[$. It is however much more difficult to prove the other way: i.e. \textit{if the MGF of a random variable is finite only on part of the real line} THEN \textit{the density (or the complementary cumulative distribution function) is exponential}.  This paper derives sufficient condition under which this last statement is true. Furthermore we show that we only need the asymptotic behavior of MGF near the critical moment to obtain the right tail asymptotics of the random variable.

We consider a real-valued random variable $Z$ whose moment generating function ($ MGF :\mu \longmapsto \mathbb{E} \; e^{\mu Z}   $) is finite only on part of the real line and explodes at some critical moment $\mu^\ast < \infty$. We investigate the link between the behavior of $MGF$ near $\mu^\ast$ and the (right) tail distribution of $Z$. In \cite{benaimfriz08}, Benaim and Friz  develop criteria, checkable by looking at the moment generating function of $Z$, which guarantees that  $   \ln \mathbb{P} ( Z> x )  \sim Ê- \mu^\ast  x $, or equivalently 
\begin{equation}\label{equivBenaimFriz}
\lim_{x \to \infty} \frac{\ln \mathbb{P} ( Z> x ) }{x} =Ê- \mu^\ast.
\end{equation}
 A similar statement, but under stronger conditions, appears in \cite{nag07}.  The purpose of this paper is to show that a sharp asymptotic formula of the form 
 \begin{equation}
 \limsup_{x \to \infty} \frac{\ln \mathbb{P} ( Z> x ) - h(x) }{\log (x)} =Êc
 \end{equation}
 can be derived \textit{by only looking at the moment generating function of $Z$ near $\mu^\ast$}.

The link between the asymptotics of a function $f$ (which here represents the density or the complementary cumulative distribution of the random variable) and its integral transforms (Laplace-Stieltjes transform for example : $\hat{f} (x) :=\int_0^\infty e^{-s u} df(u) $) is the main purpose of Abelian (: results in which we pass from a function to its integral) and Tauberian (from integrals to functions) theorems. For example, Karamata's Taberian theorem gives the equivalence between $f$ being regularity varying and $\hat{f} (1/.)$ being so, under some suitable conditions known as "Tauberian conditions" (see eg. \cite{Bingham87} Theorem~4.12.7). For the Tauberian theorems regarding the moment generating functions, Kasahara's Tauberian theorem (cf. \cite{Bingham87} Theorem~4.12.1) links the regular variation property of $\log \mathbb{E} e^{x Z}$ with that of $\log \mathbb{P} (Z>x)$,  as $x \to \infty$, but it is only applicable when the moment generating function is finite for every real number. In fact, the oldest results we could find in the vast literature of Tauberian theory that relates the tail behavior of $Z$ to the behavior of $MGF$ near $\mu^\ast$ are the relatively recent works \cite{benaimfriz08} and \cite{nag07} mentioned above.

Our study is motivated by applications in finance, in particular by the study of models for financial securities whose returns have a known moment generating functions, but whose distribution functions are unknown. It is now well known that  $MGF$ explode for a large class of distributions used for modeling returns in finance (see eg. \cite{piterbarg07}). On the other hand, Lee's celebrated moment formulae \cite{lee04} establishes  an explicit link between $\mu^\ast$ and the large-strike asymptotics of the Black-Scholes implied volatility. And more recently, Benaim and Friz related the asymptotic behavior of the tails of the distribution functions of returns  to the small-strike and large-strike asymptotics of the implied volatility.

 The main part of this work is devoted to show a Tauberian Theorem relating the right tail behavior of the cumulative distribution function of $Z$ to the behavior of the $MGF$ near $\mu^\ast$.  As an application of our result, we give sharp asymptotic formulas for any superposition of a CIR process and its time-integral. Our results can also be applied to some time-changed L\'evy models, where several examples have been given in \cite{benaimfriz08}.

This paper is organized as follows: In section 2 we give the Tauberian relation between the Laplace transform of a random variable and the complementary cumulative distribution function of this variable.  In section 3 we give an application to the CIR process. We give some proofs in the appendix.

\section{Moment explosion and distribution tails}
\noindent Throughout this paper we consider a random variable $Z$ for which  there exists $\mu^\ast > 0$ such that $\mathbb{E}~ e^{ \mu Z} < \infty$, for any $\mu\in [0, \mu^\ast[$ and 
\[
\lim_{\mu \to \mu^\ast} \mathbb{E}~ e^{ \mu Z} = + \infty.
\]
Our purpose is to derive a sharp asymptotic formula for the cumulative distribution function of $Z$. An upper bound for this distribution can be easily derived using Markov argument:
\begin{equation}\label{uper_bound}
\mathbb{P} (Z  > x ) = \mathbb{P} (e^{ \mu Z } > e^{ \mu x} ) \le e^{    - \mu x + \Lambda(\mu)  }   ,
\end{equation}
whenever $\mu > 0$. It follows that
\begin{equation}\label{uperboun_markov}
\mathbb{P} (Z  > x )  \le e^{ -\sup_{0<\mu <\mu^\ast} \left(    \mu x - \Lambda(\mu)  \right)}  = e^{-\Lambda^\ast (x)},
\end{equation}
where 
\begin{equation}\label{MGF}
 \Lambda(p) := \ln \; \mathbb{E} \; e^{ p Z}  
 \end{equation}
  and  $\Lambda^\ast$ is the Fenchel-Legendre transform of $\Lambda$. We suspect this upper bound to be a good approximation of the $\mathbb{P} (Z  > x ) $.  The main purpose of this paper is to investigate how this upper bound is far from the distribution function under suitable conditions. Let's define the function
  \begin{equation}\label{PHI}
   \varphi (x) ~:  x \longmapsto   \ln \left( \mathbb{E} \; e^{(\mu^\ast - \frac{1}{x}) Z} \right) =  \Lambda(\mu^\ast - \frac{1}{x} )
  \end{equation}
    In \cite{benaimfriz08}, it's shown that if $\varphi$ is regularly varying, then 
  \[
  \ln \mathbb{P} (Z > x) \sim_{x \to \infty} - \mu^\ast x~: ~~~ \lim_{x \to  \infty} \frac{ \ln \mathbb{P} (Z > x)+ \mu^\ast x }{x} = 0.
  \]
The main result of this paper shows that if we assume that $\varphi$ is regularly varying and that $\varphi$ is sufficiently differentiable-monotone at infinity (Assumption~\ref{hyp_laregdev}) then the Fenchel-Legendre transform of $\Lambda$ gives a very sharp expansion of the  complementary cumulative distribution of $Z$.

\subsection{Regular variation theory}
Let us first recall briefly some facts from the theory of regular variation. Let $ f ~Ê: ~Ê[0, \infty[ \longrightarrow Ê[0, \infty[ $ be a measurable function. We say $f$ has regular variation of order $\alpha$ at infinity, $ U \in R_{\alpha}$, if there exists a real number $\alpha$ such that for every $x>0$,
\begin{equation}\label{regularvariation_def}
\lim_{t \to \infty} \frac{f (t x)}{f (t)} = x^\alpha.
\end{equation}
The definition may be extended to allow $\alpha = \pm \infty$ where
\[
x^{\pm \infty} = \lim_{t \to \pm  \infty} x^t.
\]
To see which function $f$ can satisfy (\ref{regularvariation_def}), we recall the representation-characterization theorem, which can be found eg. in \cite{Bingham87} (cf. Threorem~1.3.1 and Threorem~1.4.1)

\begin{theorem}[Representation-Characterization Theorem]
The function $f>0$ is measurable such that (\ref{regularvariation_def}) holds for all $x > 0$ if and only if it may be written in the form
\begin{equation}\label{representation_RV}   
f(x) = x^\alpha \exp \left( c(x) + \int_a^x \frac{\epsilon(u)}{u} d u  \right)~~(x>a)
\end{equation}
for some $a>0$, where $c(.)$ is measurable and $ c(x) \to c \in \mathbb{R} $, $ \epsilon(x) \to 0 $ as $x \to \infty$.
\end{theorem}
In particular every regularly varying function with index $\alpha$ can be written as
\[
f(x) = x^\alpha l(x),~~\textrm{where}~~ l  ~\textrm{ is slowly varying} ~~Ê(l \in R_0).
\]  
For a function $f$ defined,  locally bounded on $ [X, \infty[$, for some $X>0$, and such that $\lim_{x \to \infty} f(x) = + \infty$, the generalized inverse of $f$ 
\[
f^{\leftarrow} (y) := \inf \left\{Ê x \in [X, \infty[~ :  f (x) > y\right\} 
\]
is defined on  $ [f(X), \infty[$  and is monotone increasing to $\infty$.  In particular, if $f \in R_{\alpha}$ with $\alpha > 0$, then Theorem 1.5.12 in \cite{Bingham87} asserts that $f^{\leftarrow} \in R_\frac{1}{\alpha}$.

\subsection{Function with smooth variation }
 A positive function $g$ defined on $[X, \infty[$, for some $X$ sufficiently large,  varies smoothly with index $\alpha > 0$ (we denote $g \in SR_{\alpha}$) iff $ h(x) := \log (g(e^x))$ is $\mathcal{C^\infty}$ and 
\[
h'(x) \longrightarrow \alpha, ~~ h^{(k)} (x) \longrightarrow 0 ~~\textrm{for} ~~n = 2,3,~Ê\dots ~~\textrm{as} ~~Êx \to \infty.
\]
In particular $SR_\alpha \subset R_\alpha$. (for more of this, see \cite{Bingham87} Section 1.8). The next result will be useful later in this article
\begin{proposition}\label{derivataive_SR}
Let $\alpha > 0$ and $g \in SR_\alpha$. Then we have 
\begin{equation}
\frac{g'(x)}{ \frac{1}{x} g(x)Ê}  \longrightarrow \alpha,  ~~ \frac{g''(x)}{ \frac{1}{x^2} g(x)Ê}  \longrightarrow -\alpha + \alpha^2, ~~Ê\textrm{as} ~~Êx \to \infty.
\end{equation}
\end{proposition}

\begin{proof}
We have, by definition of $SR_\alpha$,
\[
\frac{\partial }{\partial_x} \log ( g(e^x)) = e^x \frac{g'(e^x)}{   g(e^x)Ê} = h'(x) .
\]
It follows that
\[
\frac{g'(x)}{ \frac{1}{x} g(x)Ê}  \longrightarrow \alpha,  ~~  Ê\textrm{as} ~~Êx \to \infty.
\]
Differentiating  $ g'(x) = \frac{g(x)}{x} h' (\log(x))$ we have
\begin{eqnarray*}
g''(x)  &=& \frac{g(x)}{x^2 } \left(  -h' ( \log(x)) + \frac{g'(x) x}{ g(x)} h' ( \log(x)) + h'' ( \log(x)) \right) 
\\
&=& \frac{1}{x^2} g(x) \left(   - h' (\log (x))  + h'^2 (\log (x)) + h''  (\log (x)) \right) .
\end{eqnarray*}
Thus
\[
 \frac{g''(x)}{ \frac{1}{x^2} g(x)Ê}  \longrightarrow -\alpha + \alpha^2, ~~Ê\textrm{as} ~~Êx \to \infty.
\]
\end{proof}

\subsection{New Tauberian result}
We next present our main result concerning the asymptotic behavior of the cumulative distribution of a random variable $Z$ satisfying Assumption~\ref{hyp_laregdev}.  

\begin{assumption}\label{hyp_laregdev}
There exist $ \mu^\ast > 0$ and $ \alpha  > 0$ such that 
\begin{itemize}
\item $\forall    \mu \in [0, \mu^\ast[, \; \mathbb{E} \; e^{\mu Z} < + \infty $,  
\item the function $ \varphi (x) ~:  x \longmapsto   \ln \left( \mathbb{E} \; e^{(\mu^\ast - \frac{1}{x}) Z} \right)$ is $ R_\alpha$ and   $\mathcal{C}^1 ([X, \infty[) $, for some $ X  $, sufficiently large, such that the function  $x \longmapsto x^{2} \varphi'(x) $ is $ R_{\alpha+1}$ and monotone increasing to $\infty$.
\item The function $\epsilon: x \longmapsto   [ (.)^2 \varphi'(.)]^{\leftarrow} (x)$ is smoothly varying with index $\frac{1}{\alpha + 1}$.
\end{itemize}
\end{assumption}

\noindent We emphasize that the main idea in this paper is to investigate how the upper bound (\ref{uperboun_markov}) is far from the distribution function. Let us here consider the function $ p^\ast$ defined such that: 
\[
 \Lambda^\ast ( x) = p^\ast (x) x - \Lambda(p^\ast (x)) .
 \]
   We can easily see that under Assumption~\ref{hyp_laregdev}, $ p^\ast$ can be written as
\begin{equation}\label{p_astX}
p^\ast  ~:~x \in [X, \infty[ \longmapsto {\Lambda'}^{\leftarrow} (x) =\mu^\ast - \frac{1}{ [ (.)^2 \varphi'(.)]^{\leftarrow} (x)}.
\end{equation}
Assumption~\ref{hyp_laregdev} ensures that the function $p^\ast$ is twice-differentiable and its derivatives are monotonic for large $x$. The hypothesis that $\varphi$ has regular variation appears also in  \cite{benaimfriz08} and is sufficient to derive a statement like (\ref{equivBenaimFriz}) for the distribution tail of $Z$. Here we add the hypothesis of differentiability only in order to ensure the differentiability of $p^\ast$ and to obtain the asymptotic expansion of its derivatives.

This assumption is satisfied by a large family of function with regular variation. For example we can easily check that the case  $\varphi (x) = x^\alpha$ satisfies all the criteria of Assumption~\ref{hyp_laregdev}. It can also be proved that the case  $\varphi (x) = x^\alpha + f(x)$ where $ f $ varies smoothly with index $ \beta < \alpha$  satisfies the assumption as well. In particular, in all examples treated in this paper, the function $\varphi $ takes the latter form.

We next present our main Tauberian result which relates the behavior of the the right tail behavior of the complementary cumulative distribution function of $Z$ to the behavior of the $MGF$ near $\mu^\ast$.  But first we present a kay lemma that will be crucial to prove our main theorem:
  
 \begin{lemma}\label{integral_R}
 Let Assumption~\ref{hyp_laregdev} hold  for some random variable $X$. Denote $\Lambda (p) := \ln \mathbb{E} e^{p X }Ê$ and let $x$ be large enough. Then for any $ \gamma \in \mathbb{R}$ and $\beta \in ]0,1[$, we have
 \begin{equation}
 \int_{x^{\beta-1} }^\infty  { z^{\gamma } e^{\psi_x (z) } d z } =  x^{ \frac{- \alpha }{2 (\alpha + 1) }} f (x),  
 \end{equation}
 where $f$ is slowly varying ($f \in R_0$) and
 \begin{equation}
\psi_x (z) := \left( p^\ast(x) - p^\ast(x z)  \right) x z + \Lambda( p^\ast(x z)) - \Lambda( p^\ast(x  )).
\end{equation}
  \end{lemma}

\begin{proof}
 We split the integral into two parts:
\[
I(\gamma) := \int_{R^{\beta-1} }^1  { z^{\gamma } e^{\psi_x (z) } d z }~~ \textrm{and}~~  J(\gamma) := \int_1^\infty  { z^{\gamma } e^{\psi_x (z) } d z }  .
\]
We will first show that 
\begin{equation}
I (\gamma) \sim f_1 (x)  x^{ \frac{- \alpha }{2 (\alpha + 1) }}, ~~~f_1 \in R_0  .
\end{equation}
The function  $z \in ]0,1] \longmapsto \psi_x (z)  $ is decreasing with $ \psi_x (0) = -  \Lambda( p^\ast(x  ))$ and $\psi_x (1) = 0$. In particular, for any $h \in [0,1]$, we have
\begin{eqnarray*}
 \left| I (\gamma) - \int_{1-h}^1  z^\gamma e^{\psi_x (z) } \right|&=&I (\gamma) - \int_{1-h}^1  z^\gamma e^{\psi_x (z) } = \int_{x^{\beta-1} }^{1-h}  z^\gamma e^{\psi_x (z) } 
 \nonumber\\ &\le& e^{\psi_x (1 - h) }\left( \int_{x^{\beta-1} }^{1-h}  z^\gamma d z   \right) .
\end{eqnarray*}
On the other hand, $\psi_x$ is three times differentiable and its derivatives  are given by
\begin{equation}
{\psi_x}' (z) =  \left( p^\ast(x) - p^\ast(x z)  \right) x , ~ \psi^{(n)}_x (z) = - x^n {p^\ast}^{(n-1)} (x z), ~ n =2,3.
\end{equation}
 It follows that for $h$ sufficiently small, we have
\begin{eqnarray*}
\psi_x ( 1 - h ) &=& \psi_x (1) - {\psi_x}' (1) h + {\psi_x}'' (1) \frac{h^2}{2} - {\psi_x}^{(3)} (1) \frac{h^3}{6} + \dots \nonumber\\ &=&
- x^2 {p^\ast}' (x) \frac{h^2}{2} +  x^3 {p^\ast}'' (x) \frac{h^3}{6} + \dots 
\end{eqnarray*}
In particular, we have
\begin{eqnarray}
\sup_{z \in [ 0,h]}\left| \psi_x (1- z  ) + x^2 {p^\ast}' (x) \frac{z^2}{2} \right| \le \frac{h^3}{6} x^3 \sup_{z \in [ 1 - h, 1]}|{p^\ast}'' (zx) | &= & \frac{h^3}{6} x^3 |{p^\ast}'' ( x(1-h))| .
\end{eqnarray}   
Under assumption~\ref{hyp_laregdev}, we have $p^\ast(x) = \mu^\ast -\frac{1}{ \epsilon(x)}$, where $ \epsilon (x)=  [ (.)^2 \varphi'(.)]^{\leftarrow} (x) \in SR_\frac{1}{\alpha + 1}   $. Using Proposition~\ref{derivataive_SR} we can easily show that there exist $f_0, ~ g_0 \in R_0 $ (slowly varying) such that $ x^2 {p^\ast}' (x) \sim x^\frac{\alpha}{\alpha + 1} f_0 (x)$ and $ x^3 |{p^\ast}'' (x)| \sim x^\frac{\alpha}{\alpha + 1} g_0 (x)$. Indeed as $ \epsilon \in SR_\frac{1}{\alpha + 1}  \subset R_\frac{1}{\alpha + 1}$, there exists a slowly varying function $f$ such that $ \epsilon(x) = x^\frac{1}{\alpha + 1} f(x)$. On the other hand, using Proposition~\ref{derivataive_SR} , which gives $ \frac{\epsilon' (x) }{ \frac{1}{x} \epsilon (x) }Ê \sim \frac{1}{\alpha + 1}  $ and $ \frac{\epsilon'' (x) }{ \frac{1}{x^2} \epsilon (x) }Ê \sim - \frac{1}{\alpha + 1} + (\frac{1}{\alpha + 1})^2   $,  we have\[
{p^\ast}' (x) = \frac{\epsilon' (x) }{\epsilon^2 (x) } = \frac{\epsilon' (x) }{ \frac{1}{x} \epsilon (x) }Ê ~\frac{1}{x \epsilon (x)}  \sim \frac{1}{\alpha + 1}  \frac{1}{x \epsilon (x)} =  \frac{1}{x^{1 + \frac{1}{\alpha + 1}Ê}} \frac{1}{ (\alpha + 1) f(x)} , 
\]
where the function $ f_0(x) := \frac{1}{ (\alpha + 1) f(x)}$ is slowly varying. Similarly, we have  
\[
{p^\ast}'' (x) = \frac{ \epsilon''  }{ \epsilon^2} -  2 \frac{  {\epsilon'}^2}{ \epsilon^3} =\frac{ 1  }{ x^2 \epsilon}(  \frac{ \epsilon''  }{ \frac{1}{x^2}  \epsilon} - 2 (  \frac{  {\epsilon'}}{ \frac{1}{x}Ê\epsilon} )^2) \sim c \frac{ 1  }{ x^2 \epsilon},
\]
where $ c = - \frac{1}{\alpha + 1} + (\frac{1}{\alpha + 1})^2 - 2  (\frac{1}{\alpha + 1})^2 =   -\frac{1}{\alpha + 1} - (\frac{1}{\alpha + 1})^2 $. Hence
\begin{eqnarray}\label{taylor}
\sup_{z \in [ 0,h]}\left| \psi_x ( 1-z  ) + x^2 {p^\ast}' (x) \frac{z^2}{2} \right| \le c_1 g_0(x) \frac{ h^3}{3} x^\frac{\alpha}{\alpha  + 1},
\end{eqnarray}
with a  constant $c_1 >0$.  Let's  set
\begin{equation}
h \equiv h (x)  = \frac{x^{- \frac{5}{12} \frac{\alpha}{\alpha + 1} }}{   (2 f_0 (x))^\frac{1}{2}}.
\end{equation}
For this particular $h$, we have
\[
\sup_{z \in [ 0,h]}\left| \psi_x ( 1-z  ) + x^2 {p^\ast}' (x) \frac{z^2}{2} \right| \le c_1 x^{- \frac{1}{4} \frac{\alpha}{\alpha + 1} } \frac{g_0 (x) }{(f_0 (x))^\frac{3}{2}}.
\]
It follows that
\[
\psi_x (1 - h)  = -x^2 {p^\ast}' (x) \frac{h^2}{2} + \mathcal{O}_0( x^{- \frac{1}{4} \frac{\alpha}{\alpha + 1} } \frac{g_0 (x) }{(f_0 (x))^\frac{3}{2}})
 \sim  - x^{  \frac{1}{6} \frac{\alpha}{\alpha + 1} },
\] 
where $| \mathcal{O}_0 (x) | \ll  |x|$ for $x$ sufficiently small. i.e $\lim_{x \to 0} \frac{ \mathcal{O}_0 (x)  }{  x }= 0$.  It follows that
\begin{eqnarray*}
 \left| I (\gamma) - \int_{1-h}^1  z^\gamma e^{\psi_x (z) } \right|  &=&     \int_{x^{\beta - 1}}^{1-h}   z^\gamma e^{\psi_x (z) }  \le
  e^{ \psi_x (1 - h) }\left( \int_{x^{\beta-1} }^{1-h}  z^\gamma d z   \right)
  \\ &\le&
  e^{-  x^{  \frac{1}{6} \frac{\alpha}{\alpha + 1} } }\left( \int_{x^{\beta-1} }^{1-h}  z^\gamma d z   \right) \le Q (x) e^{- x^{  \frac{1}{6} \frac{\alpha}{\alpha + 1} } },
\end{eqnarray*}
for some polynomial function $Q$. On the other hand, we have, using (\ref{taylor}),
\[
\left| \int_{1-h}^1  z^\gamma e^{\psi_x (z) } -  \int_{1-h}^1  z^\gamma e^{- x^2 {p^\ast}' (x) \frac{(1-z)^2}{2} } d z   \right| \le
\int_{1-h}^1  z^\gamma e^{- x^2 {p^\ast}' (x) \frac{(1-z)^2}{2} } d z  \left| e^{ \frac{h^3}{6} x^3 |{p^\ast}'' ( x(1-h))| } - 1 \right| .
\]
We emphasize that $ \frac{h^3}{6} x^3 |{p^\ast}'' ( x(1-h))| =  x^{- \frac{1}{4} \frac{\alpha}{\alpha + 1} }  g_1 (x) $ with $g_1 \in R_0$. It follows that 
\begin{eqnarray*}
I (\gamma )& \sim & \int_{1-h}^1  z^\gamma e^{\psi_x (z) }  \sim  \int_{1-h}^1  z^\gamma e^{- x^2 {p^\ast}' (x) \frac{(1-z)^2}{2} } d z    \nonumber\\&=&
\int_0^{ h(x)} (1 - z)^\gamma   e^{- x^2 {p^\ast}' (x) \frac{z^2}{2} } d z \sim  \int_0^{ h(x)}    e^{- x^2 {p^\ast}' (x) \frac{z^2}{2} } d z  \\
&\sim& \frac{1}{\sqrt{ x^2 {p^\ast}' (x)}} \int_0^{ h(x)\sqrt{ x^2 {p^\ast}' (x)} } e^{ - z^2/2} d z.
\end{eqnarray*}
Here $f (x) \sim g (x)$ means that $ \lim_{x \to \infty} \frac{ f(x)}{g(x)} = 1Ê$. Now by definition of $h$ we have for $x$ sufficiently large
\[
h(x) \sqrt{ x^2 {p^\ast}' (x)} =   \frac{x^{- \frac{5}{12} \frac{\alpha}{\alpha + 1} }}{   (2 f_0 (x))^\frac{1}{2}} \times x^\frac{\alpha}{2(\alpha + 1)} \sqrt{ f_0 (x)} = x^{  \frac{1}{12} \frac{\alpha}{\alpha + 1} } .
\]
Hence $ h(x) \sqrt{ x^2 {p^\ast}' (x)}\longrightarrow + \infty ~~\textrm{as}~~ x \to \infty  $. Thus
\begin{equation}
I(\gamma) \sim  \frac{1}{\sqrt{ x^2 {p^\ast}' (x)}} = x^\frac{-\alpha}{2(\alpha + 1)}  \frac{1}{\sqrt{ f_0 (x)} }.
\end{equation}
 
We can show by the same way that
\begin{equation}
J(\gamma) := \int_1^\infty   z^{\gamma } e^{\psi_x (z) } d z =  x^{ \frac{-\alpha}{2 (\alpha + 1 )}}g(x), ~~~~\textrm{whee}~~~g \in R_0.
\end{equation}

  \end{proof}

 The main Tauberian theorem of this paper is the following:
 
\begin{theorem}{\label{tauberian_theorem}} %
Let Assumption~\ref{hyp_laregdev} hold for some random variable $Z$. Denote $\Lambda (p) = \ln \mathbb{E} e^{ p Z}$ and $\Lambda^\ast$ the Fenchel-Legendre transform of $\Lambda$. Then we have
 \begin{eqnarray}\label{logtail}
 \limsup_{ x \to \infty}  \frac{ \ln \mathbb{P} ( Z  > x ) + \Lambda^\ast (x) }{\ln(x)}  &\in& \left[    - \frac{  \alpha + 2}{2 (\alpha+1)}    ,0\right] ,  \\
  \liminf_{ x \to \infty}  \frac{ \ln \mathbb{P} ( Z  > x ) + \Lambda^\ast  (x) }{\ln(x)}  &\le &     - \frac{  \alpha + 2}{2 (\alpha+1)}  .
 \end{eqnarray}
 \end{theorem}

\begin{proof}
The upper bound is obtained using (\ref{uperboun_markov}), as a consequence of Markov's inequality: 
\begin{equation}\label{upper_bound}
\limsup_{ x \to \infty} \frac{ \ln \mathbb{P} (Z  > x ) + \Lambda^\ast (x) }{\ln x } \le  0.
  \end{equation} 

\paragraph*{{\bf Lower Bound for 'limsup':}}  
In order to derive a lower bound for the $ \limsup$ in (\ref{logtail}) we proceed as follows. We suppose that
\[
 \limsup_{ x \to \infty}  \frac{ \ln \mathbb{P} ( Z  > x ) + \Lambda^\ast (x) }{\ln(x)}   = \nu^\ast <    - \frac{  \alpha + 2}{2 (\alpha+1)}   .
\]
We will show that this will lead to a contradiction.

First, we can easily see that $e^{pZ}$ can be written in terms of any $c>0$ as
\begin{equation}
e^{p Z} =e^{p (c \wedge Z)  } + p  1_ { Z > c }\int_{c}^Z e^{p k} d k .
\end{equation}
It follows that for any $p < \mu^\ast$,  $\beta \in ]0, 1[$ and  $x$ sufficiently large,
\begin{eqnarray*}
 \mathbb{E} \left(  e^{ p Z } \right) &=&  \mathbb{ E}  (e^{p (x^\beta\wedge Z)  } ) +  p  \mathbb{E}  \int_{x^\beta}^\infty  e^{p k} 1_{ Z > k} d k  \nonumber\\
  &=&  \mathbb{ E}  (e^{p (x^\beta\wedge Z)  } ) +  p \int_{x^\beta}^\infty  {   e^{p k}  \mathbb{P}(Z > k)  d k  }
  \nonumber\\ &\le &   e^{p x^\beta}+ p\int_{x^\beta}^\infty  {   e^{p k} ~ e^{- \Lambda^\ast(k) + \nu' \ln (k) }  d k  }   ,
  \end{eqnarray*}
  whenever $ \nu' \in ] \nu^\ast,  - \frac{  \alpha + 2}{2 (\alpha+1)}    [$.  Choosing  $p  = p^\ast (x) $ we have
  \[
\mathbb{E} \left(  e^{ p^\ast (x)  Z } \right) \le       e^{ p^\ast (x) x^\beta}+ p^\ast (x)e^{\Lambda ( p^\ast (x)) } x^{\nu' + 1} \int_{x^{\beta -1}}^\infty  { z^{\nu'} e^{\psi_x (z) } d z },
  \]
 where
 \begin{equation}
 \psi_x (z) := \left( p^\ast(x) - p^\ast(x z)  \right) x z + \Lambda( p^\ast(x z)) - \Lambda( p^\ast(x  )).
 \end{equation}
 Now  for $\beta = \frac{\alpha}{2(1+\alpha)}$ and using Lemma~\ref{integral_R}  we have, for $x$ sufficiently large,
 \[
 \int_{x^{\beta-1} }^\infty  { z^{\gamma } e^{\psi_x (z) } d z } =  x^{ \frac{- \alpha }{2 (\alpha + 1) }} f (x),  
 \]
 where $f$ is slowly varying. Hence
  \begin{eqnarray*}
e^{ \Lambda (p^\ast (x))} &\le&      e^{ p^\ast (x) x^\beta}+p^\ast (x)  f (x)  \; e^{\Lambda ( p^\ast (x)) } x^{\nu' + 1} x^{ \frac{- \alpha }{2 (\alpha + 1) }}  .
\end{eqnarray*}   
 On the other hand we have   
\begin{equation}\label{LambdaOfp_ast}
\lim_{x \to \infty } \frac{ p^\ast (x) x^\beta }{\Lambda ( p^\ast (x))} = 0.
\end{equation}
Indeed,  using (\ref{p_astX}) we have
\[
p^\ast   (x) =  {\Lambda'}^{\leftarrow} (x) =\mu^\ast - \frac{1}{ [ (.)^2 \varphi'(.)]^{\leftarrow} (x)},
\]
with $\varphi \in R_\alpha$. We emphasize that $\varphi(x) := \Lambda(\mu^\ast - \frac{1}{x}) $. It follows that
\[
\Lambda(p^\ast   (x) ) = \varphi \left(  [ (.)^2 \varphi'(.)]^{\leftarrow} (x) \right).
\] 
Assumption~\ref{hyp_laregdev} implies that $ \{x \longrightarrow x^2 \varphi'(x) \} \in R_{\alpha + 1} $ and  $ [ (.)^2 \varphi'(.)]^{\leftarrow} (.) \in R_{\frac{1}{\alpha + 1}Ê}$. It follows that $\Lambda(p^\ast   (.) ) \in  R_{\frac{\alpha}{\alpha + 1}Ê}$ and hence $ \lim_{x \to \infty } \frac{ p^\ast (x) x^\beta }{\Lambda ( p^\ast (x))} = 0$.

 Now using (\ref{LambdaOfp_ast}) we have, for $x$ sufficiently large,
\[
e^{ \Lambda (p^\ast (x))}  \le 2 f(x) \; e^{\Lambda ( p^\ast (x)) } x^{\nu' + 1} x^{ \frac{- \alpha }{2 (\alpha + 1) }}.
\]
Which is impossible as $ \nu' <   - \frac{  \alpha + 2}{2 (\alpha+1)}     = -1 + \frac{ \alpha }{2 (\alpha + 1) }$. Thus
 \begin{equation}
 \limsup_{ x \to \infty}  \frac{ \ln \mathbb{P} ( Z  > x ) + \Lambda^\ast (x) }{\ln(x)}  \geq   - \frac{  \alpha + 2}{2 (\alpha+1)}.
\end{equation}

\paragraph*{{\bf Upper Bound for 'liminf':}}
 Suppose that
\[
 \liminf_{ x \to \infty}  \frac{ \ln \mathbb{P} ( Z  > x ) + \Lambda^\ast (x) }{\ln(x)}   = \nu_\ast >    - \frac{  \alpha + 2}{2 (\alpha+1)}   .
\]
 We obtain, by the same way as before,  
\[
 e^{ \Lambda (p^\ast (x))} =\mathbb{E} \left(  e^{ p^\ast (x)  Z } \right) \geq p^\ast (x) e^{\Lambda ( p^\ast (x)) } x^{\nu' + 1} \int_{x^{\beta -1}}^\infty  { z^{\nu'} e^{\psi_x (z) } d z } \sim C \;  f(x)e^{\Lambda ( p^\ast (x)) } x^{\nu' + 1} x^{ \frac{- \alpha }{2 (\alpha + 1) }},
  \]  
whenever $ \nu' \in ] - \frac{  \alpha + 2}{2 (\alpha+1)} , \nu_\ast[$, which is yet impossible. Thus 
\begin{equation}
 \liminf_{ x \to \infty}  \frac{ \ln \mathbb{P} ( Z  > x ) + \Lambda^\ast (x) }{\ln(x)}   \le    - \frac{  \alpha + 2}{2 (\alpha+1)}  .
\end{equation}  
\end{proof}

\section{Application: CIR process}
\noindent  Theorem~ \ref{tauberian_theorem} applies to several examples where the risk neutral log-price is modeled by L\'evy process such as the variance gamma model whose moment generating function is given as $MGF(p) = \left(  \frac{g m }{ (m-p)(p+g)}  \right)^c$. This function explodes at $\mu^\ast = m$ and for $x$ small enough, $MGF (\mu^\ast - x) = \left(  \frac{g m }{  m-x+g }  \right)^c \frac{1}{x^c} $. Hence Assumption~\ref{hyp_laregdev} holds here. Our result applies also to the NIG model as well as for double exponential model (see section 5 in \cite{benaimfriz08} for more details).

 In this section, we study the asymptotic behavior of the cumulative distribution of a superposition of the CIR process and its time integral.  The CIR process has been proposed by Cox, Ingersoll and Ross \cite{cir85} to model the evolution of interest rates. It is defined as the unique solution of the following stochastic differential equation:
\begin{equation}\label{eq-CIR}
dV_t=(a-bV_t)dt+\sigma\sqrt{V_t}dW_t,~~ V_0 = v,
\end{equation}
where $a, \;\sigma, \; v \;\geq 0$ and $b\in\mathbb{R}$ (see \cite{ikida89} for the existence and uniqueness of the SDE). The density transition $V_t $ is known and is given by (cf. \cite{lamberton97})
\begin{equation}\label{density-CIR}
 P_t (v,z) = \frac{  e^{ b t }  }{ 2c_t} \left(  \frac{ z e^{ b t } }{v} \right)^{  \frac{\nu -1}{2}} \exp \left( - \frac{ v + z e^{ b t } }{2 c_t}  \right) B_{\nu } \left( \frac{1}{c_t} \sqrt{ v z e^{ b t } } \right), 
 \end{equation}
 where $ c_t = \frac{(e^{ b t } - 1) \sigma^2  }{4 b} $, $\nu =  2 a/\sigma^2  $ and $B_\nu$ is the modified Bessel function of order $\nu$. 
We set
\begin{equation}
I_t = \int_0^t V_u d u .
\end{equation}
Gulisashvili and Stein \cite{stein10} (cf Lemma 6.3 and Remark 8.1) show that $I_t$ admits a distribution density (see also Dufresne \cite{dufr01} for a similar result). A sharp asymptotic formulas for the distribution density of $I_t$ is obtained in \cite{stein10} (cf Theorem 2.4) as
\begin{equation}\label{density-CIRI}
p^I_t (y) =  A_t \; e^{ - C_t  y +  B_t  \sqrt{y}  } \;  y^{  \frac{\nu}{2}  - \frac{3}{4} } \left( 1 + \mathcal{O} ( y^{ -    \frac{1}{2}}) \right),
\end{equation}
where $A$, $B$ and $C $ are positive constants.  A similar result exits also for the cumulative distribution of the stock price in Heston model (see Friz et al \cite{friz11}). We will show that a similar formula can be obtained for any superposition of the CIR process and its time integral.

\subsection{Moment generating function of a superposition of $V$ and $I$:}
To compute the moment generating function of a superposition of $V$ and $I$, we use the additivity property of $V$ with respect to $(a,v)$. Indeed, if we fix parameters $b, ~Ê\sigma$ and denote by $V^{v,a}_t$the solution of \eqref{eq-CIR}  corresponding to parameters $(a, b, v, \sigma)$  in which $W$ is replaced by an independent Brownian motion, then for any $(v_1,v_2)$, $(a_1,a_2)\in\mathbb{R}_+^2$, the process $V^{v_1+v_2,a_1+a_2}$ has the same law as $(V^{v_1,a_1}_t+V^{v_2,a_2}_t)$. This additivity property of CIR processes is well known; it follows from the fact that a CIR process is a squared Bessel process subjected to a deterministic change of time, and from the additivity property of squared Bessel processes, which is a celebrated result of Shiga and Watanabe (see \cite{shiga73}). As consequence of this property, if we define $ F^a_{\lambda_1, \lambda_2 }(t,v)$, for $\lambda_1$ and $\lambda_2$ $ \in \mathbb{R}$, by
 \[
 F^a_{\lambda_1, \lambda_2 }(t,v)=\mathbb{E} ~e^{ \lambda_1 V^{v,a}_t + \lambda_2   I_t  },
 \]
 we have 
 $F^{a_1+a_2}_{\lambda_1, \lambda_2 } (t,v_1+v_2)=F^{a_1}_{\lambda_1, \lambda_2 }(t,v_1)F^{a_2}_{\lambda_1, \lambda_2 }(t,v_2)$.
 It follows that the function $F^a_{\lambda_1, \lambda_2 }$ is given by
  \[
  F^a_{\lambda_1, \lambda_2 }(t,v)=e^{a \varphi_{\lambda_1, \lambda_2 }(t)+v\psi_{\lambda_1, \lambda_2 }(t)},
  \]  
where the functions $  \varphi_{\lambda_1, \lambda_2 }(t)$ and $ \psi_{\lambda_1, \lambda_2 } (t)$ do not depend on $a$ and $v$. Under the assumption ($\lambda_2 \le \frac{b^2}{2 \sigma^2} $ and $2 a \geq \sigma^2 $),  an explicit formulas for the Laplace transform of the joint distribution of a CIR process and its time-integral is obtained in \cite{hurd08}. The next result gives the explicit formulas for $\varphi$ and $\psi$ without any restrictions on the parameters of the CIR process.

\begin{theorem}\label{moment_esponentiels}
For $ \lambda_1 , \; \lambda_2\; \in \mathbb{R}$,  let $ \psi_{\lambda_1, \lambda_2}$ be the maximal solution of
\begin{equation}\label{eq-psi}
 \left\{\begin{array}{l}
 \displaystyle \psi'_{\lambda_1, \lambda_2 }(t)=  \frac{\sigma^2}{2}\left(
    \psi^2_{\lambda_1, \lambda_2 }(t)-2\frac{b}{\sigma^2}\psi_{\lambda_1, \lambda_2 }(t)
      +2\frac{\lambda_2}{\sigma^2}  \right),\\ 
      \psi_{\lambda_1, \lambda_2 }(0) = \lambda_1,
         \end{array}
           \right.
 \end{equation}
 defined over $ [0, t^\ast_{\lambda_1, \lambda_2} [$, with $ t^\ast_{\lambda_1, \lambda_2} \in ]0,\infty]$ such that if $   t^\ast_{\lambda_1, \lambda_2}<\infty$ then $ \lim_{t \to t^\ast_{\lambda_1, \lambda_2}} \psi_{\lambda_1, \lambda_2 }(t) = + \infty$. Also let $ \varphi_{\lambda_1, \lambda_2}$ be the time integral of $\psi_{\lambda_1, \lambda_2 }(.)$: $ \varphi_{\lambda_1, \lambda_2} (t) := \int_0^t \psi_{\lambda_1, \lambda_2 } (u) d u $. Then, for any $ T < t^\ast_{\lambda_1, \lambda_2}$, we have
 \begin{equation}
 \mathbb{E} \; e^{ \lambda_1 V^{v,a}_T + \lambda_2  I_T  } \;  = \; e^{a\varphi_{\lambda_1, \lambda_2 }(T)+v\psi_{\lambda_1, \lambda_2 }(T)}
 \end{equation}
\end{theorem}  
\begin{proof}
   Let's define $F^a_{\lambda_1, \lambda_2 }(t,v)$ by
 \[
  F^a_{\lambda_1, \lambda_2 }(t,v)=e^{a\varphi_{\lambda_1, \lambda_2 }(t)+v\psi_{\lambda_1, \lambda_2 }(t)}.
 \]
The function  $F^a_{\lambda_1, \lambda_2 }$ is, by definition, the maximal solution of
 \[
 \left\{\begin{array}{l}
 \displaystyle\frac{\partial F^a_{\lambda_1, \lambda_2 }}{\partial t}(t,v)=
      \frac{\sigma^2}{2}v\frac{\partial^2 F^a_{\lambda_1, \lambda_2 } }{\partial v^2}(t,v)
      +(a-bv)\frac{\partial F^a_{\lambda_1, \lambda_2 }}{\partial v}(t,v)+ \lambda_2 v F^a_{\lambda_1, \lambda_2 }(t,v),~~ 0 < t \le T,\\ \\
      F^a_{\lambda_1, \lambda_2 }(0,v)=e^{  \lambda_1 v}.
         \end{array}
           \right.
 \]
 This is indeed equivalent to  $\varphi_{\lambda_1, \lambda_2 }(0)=0$, $ \psi_{\lambda_1, \lambda_2 }(0)= \lambda_1$ and
 \begin{eqnarray*}
 a\varphi_{\lambda_1, \lambda_2 }'(t)+v\psi'_{\lambda_1, \lambda_2 }(t)
 &=&
      \frac{\sigma^2}{2}v\psi^2_{\lambda_1, \lambda_2 }(t)+(a-bv)\psi_{\lambda_1, \lambda_2 }(t)+  \lambda_2   v\\
      &=&a\psi_{\lambda_1, \lambda_2 }(t)+v\left(\frac{\sigma^2}{2}\psi^2_{\lambda_1, \lambda_2 }(t)
                 -b\psi(t)+  \lambda_2 \right).
\end{eqnarray*}
This system is equivalent to  $\varphi_{\lambda_1, \lambda_2 }'(t)=\psi_{\lambda_1, \lambda_2 }(t)$ and
 \[
\psi'_{\lambda_1, \lambda_2 }(t)=  \frac{\sigma^2}{2}\left(
    \psi^2_{\lambda_1, \lambda_2 }(t)-2\frac{b}{\sigma^2}\psi_{\lambda_1, \lambda_2 }(t)
      +2\frac{\lambda_2}{\sigma^2}  \right).
\]
Hence for any $ T < t^\ast_{\lambda_1, \lambda_2}$, the process $\left( e^{  \lambda_2 \int_0^s V^{v,a}_u du}  F^a_{\lambda_1, \lambda_2 }(T -s,V^{v,a}_s)  \right)_{ s \le T}$ is a positive local martingale; it is therefore a super-martingale. In particular, for any stopping time $\tau \in \mathcal{T}_{0,T}$, we have
\[
E \; e^{  \lambda_2 \int_0^\tau V^{v,a}_u du }F^a_{\lambda_1, \lambda_2 }(T -\tau,V^{v,a}_\tau) \le 
F^a_{\lambda_1, \lambda_2 }(T  ,V^{v,a}_0) .
\]
Hence
\begin{equation}
\sup_{\tau \in \mathcal{T}_{0,T} }E \; e^{  \lambda_2 \int_0^\tau V^{v,a}_u du + a \varphi_{\lambda_1, \lambda_2} (T - \tau) + V^{v,a}_\tau \psi_{\lambda_1, \lambda_2} (T - \tau) } \le 
e^{a\varphi_{\lambda_1, \lambda_2 }(T)+v\psi_{\lambda_1, \lambda_2 }(T)}.
\end{equation}
We will show that there exists $p > 1$ such that
\[
\sup_{\tau \in \mathcal{T}_{0,T} }E \; e^{ p \left( \lambda_2 \int_0^\tau V^{v,a}_u du + a \varphi_{\lambda_1, \lambda_2} (T - \tau) + V^{v,a}_\tau \psi_{\lambda_1, \lambda_2} (T - \tau)  \right)} < + \infty.
\]
We will first show that  $ \psi_{\lambda_1, \lambda_2}$ is increasing with respect to $ \lambda_1$ and $\lambda_2$. This means that if $ \lambda_1 < \bar{\lambda}_1$ (resp  $ \lambda_2 < \bar{\lambda}_2$), we have $\psi_{\lambda_1, \lambda_2} (t) \le \psi_{  \bar{\lambda}_1, \lambda_2} (t)$, for every $ t <   t^\ast  ( \bar{\lambda}_1, \lambda_2)$ (resp $\psi_{\lambda_1, \lambda_2} (t) \le \psi_{\lambda_1, \bar{\lambda}_2} (t)$,  whenever $   t <   t^\ast  (\lambda_1,  \bar{\lambda}_2 )$). For this, we define the functions $f$ and $g$ by  $ f =  \psi_{  \bar{\lambda}_1, \lambda_2} -\psi_{\lambda_1, \lambda_2}   $ and $ g = \psi_{\lambda_1, \bar{\lambda}_2} - \psi_{\lambda_1, \lambda_2} $. We have
\begin{eqnarray*}
f(0) &=& \bar{\lambda}_1 - \lambda_2 > 0, ~~Êg(0) = 0, \nonumber\\ 
f' (t)& =& \frac{\sigma^2}{2} \left(\psi^2_{  \bar{\lambda}_1, \lambda_2} (t) - \psi^2_{  \lambda_1, \lambda_2}  - \frac{2 b}{\sigma^2} f (t) \right) =
\frac{\sigma^2}{2} f(t) \left(\psi_{  \bar{\lambda}_1, \lambda_2} (t) + \psi_{  \lambda_1, \lambda_2}  - \frac{2 b}{\sigma^2}   \right)
\end{eqnarray*}
and
\begin{eqnarray*}
g' (t)& =& \frac{\sigma^2}{2} \left(\psi^2_{  \bar{\lambda}_1, \lambda_2} (t) - \psi^2_{  \lambda_1, \lambda_2}  - \frac{2 b}{\sigma^2} g (t)  + 2 \frac{\bar{\lambda}_2 - \lambda_2}{\sigma^2} \right) \geq
\frac{\sigma^2}{2} g(t) \left(\psi_{  \bar{\lambda}_1, \lambda_2} (t) + \psi_{  \lambda_1, \lambda_2}  - \frac{2 b}{\sigma^2}   \right).
\end{eqnarray*}
The functions $f$ and $g$ satisfy $ (f e^{A})' (t) = 0$, with $ f (0) > 0$ and $ (g e^{A} )' (t) \geq 0$, with $ g (0) = 0$, where 
\[
 A' (t) = \frac{\sigma^2}{2}  \left(\psi_{  \bar{\lambda}_1, \lambda_2} (t) + \psi_{  \lambda_1, \lambda_2}  - \frac{2 b}{\sigma^2}   \right).
 \]
It follows that $ f \geq 0$ and $ g \geq  0$. Hence
\[
\psi_{  \bar{\lambda}_1, \lambda_2} (t) - \psi_{  \lambda_1, \lambda_2} (t) = (\bar{\lambda}_1 -  \lambda_1) \; e^{bt- \int_0^t  (\psi_{ \bar{\lambda}_1, \lambda_2}  + \psi_{  \lambda_1, \lambda_2}   )(s) d s }, 
\] 
whenever $ t < t^\ast(\bar{\lambda}_1, \lambda_2)$. It follows that for every $T < t^\ast_{\lambda_1, \lambda_2}$ and for  $   \epsilon  >0 $ sufficiently small, there exists $ \lambda^\epsilon_1 > \lambda_1$ such that
\[
(1 + \epsilon) \psi_{\lambda_1, \lambda_2} (t) \le \psi_{\lambda^\epsilon_1, \lambda_2} (t), ~~ \forall t \le T.
\] 
So for any stopping time $ \tau \in \mathcal{T}_{0,T}$,
\[
E \; e^{ (1 + \epsilon) \left( \lambda_2 \int_0^\tau V^{v,a}_u du + a \varphi_{\lambda_1, \lambda_2} (T - \tau) + V^{v,a}_\tau \psi_{\lambda_1, \lambda_2} (T - \tau)  \right)} \le 
e^{a_\epsilon \varphi_{\lambda^\epsilon_1, \lambda^\epsilon_2 }(T)+v\psi_{\lambda^\epsilon_1, \lambda^\epsilon_2 }(T)},
\]
where $ a_\epsilon = (1 + \epsilon) a$ and $ \lambda^\epsilon_2 = (1 + \epsilon) \lambda_2$. Thus
\begin{equation}\label{sup_tau}
\sup_{\tau \in \mathcal{T}_{0,T} }E \; e^{ (1 + \epsilon) \left( \lambda_2 \int_0^\tau V^{v,a}_u du + a \varphi_{\lambda_1, \lambda_2} (T - \tau) + V^{v,a}_\tau \psi_{\lambda_1, \lambda_2} (T - \tau)  \right)} < + \infty.
\end{equation}

In conclusion, the local martingale  $\left( M_s :=  e^{  \lambda_2 \int_0^s V^{v,a}_u du}  F^a_{\lambda_1, \lambda_2 }(T -s,V^{v,a}_s)  \right)_{ s \le T} $ is uniformly integrable; it is therefore a true martingale. In particular, we have
\[
E \; e^{   \lambda_2 \int_0^t V^{v,a}_u du + a \varphi_{\lambda_1, \lambda_2} (T - t) + V^{v,a}_t \psi_{\lambda_1, \lambda_2} (T - t)   } = 
e^{     a \varphi_{\lambda_1, \lambda_2} (T  ) + V^{v,a}_t \psi_{\lambda_1, \lambda_2} (T  )   },  
\]  
whenever $t \le T $.Ê Thus
\[
\mathbb{E} \; e^{   \lambda_2 \int_0^T V^{v,a}_u du  + \lambda_1  V^{v,a}_T     } = 
e^{     a \varphi_{\lambda_1, \lambda_2} (T  ) + V^{v,a}_t \psi_{\lambda_1, \lambda_2} (T  )   }.
\] 
 \end{proof}

 \subsection{Moment explosion}
 The function $ F^a_{\lambda_1, \lambda_2} $ is defined over $ [0, t^\ast_{\lambda_1, \lambda_2} [$ so that \\ $ \lim_{ t \to t^\ast_{\lambda_1, \lambda_2} }F^a_{\lambda_1, \lambda_2}(t,v) = + \infty$ if $t^\ast_{\lambda_1, \lambda_2} <\infty$. For $ \lambda_1, \; \lambda_2 \; \in \mathbb{R}$, let us set $ t^\ast_{\lambda_1, \lambda_2} (\mu) = t^\ast_{\mu \lambda_1, \mu \lambda_2}$.  Let us also define $\mu^\ast_- (t)$ and $ \mu^\ast_+ (t)$ by
 \begin{equation}
  \mu^\ast_- (t) = -\sup\left\{ \mu < 0: ~~ t^\ast_{\lambda_1, \lambda_2} ( \mu) = t \right\} ~~Ê\textrm{and} ~~\mu^\ast_+ (t) = \inf\left\{ \mu > 0: ~~ t^\ast_{\lambda_1, \lambda_2} (\mu) = t \right\},
 \end{equation} 
  with $\sup \emptyset = -   \infty$ and $\inf \emptyset =   + \infty$. The moment generating function of $\lambda_1 V_t +   \lambda_2 I_t$ is defined over $]-\mu^\ast_- (t) , \mu^\ast_+ (t)[$ and if $ \mu^\ast_- (t) <   \infty$ (resp $ \mu^\ast_+ (t)< + \infty$) then 
  \[ \lim_{\mu \to   \mu^\ast_- (t)}  \mathbb{E} \; e^{ -\mu (\lambda_1 V_t +  \lambda_2 I_t)} = + \infty~ (\textrm{resp.}~ \lim_{\mu \to   \mu^\ast_+ (t)}  \mathbb{E} \; e^{ \mu (\lambda_1 V_t +  \lambda_2 I_t)}= + \infty ).
  \]
 When either $\mu^\ast_-(t)$ or $\mu^\ast_+(t)$ is finite, the behavior of the moment generating function near the critical moment is given in the next result, whose proof can be found in the appendix.

\begin{theorem}\label{equiv_laplace}
Let $t > 0$ and $ \lambda_1, \lambda_2 \in \mathbb{R}$ such that $\max (\lambda_1, \lambda_2) > 0$ (resp. $\min (\lambda_1, \lambda_2) < 0$). Then $\mu^\ast_+ (t) < + \infty$ (resp. $\mu^\ast_- (t) < + \infty$). Furthermore, $ \exists \omega^+_t > 0$ (resp. $ \exists \omega^-_t > 0$) such that 
\[
 \mu \longmapsto   \ln \left( \mathbb{E} \; e^{ \mu ( \lambda_1 V_t +  \lambda_2 \int_0^t{ V_u d u)}} \right)- H(\mu)~~ \left( \textrm{resp.} ~~Ê\mu \longmapsto   \ln \left( \mathbb{E} \; e^{- \mu ( \lambda_1 V_t +  \lambda_2 \int_0^t{ V_u d u)}} \right)- H(\mu)\right)
\]
 is bounded near $  \mu^\ast_+ (t)$ (resp. $  \mu^\ast_- (t)$ ), where $H (\mu) =  \frac{\omega^+_t}{ \mu^\ast_+ (t) - \mu}  - \frac{2 a }{\sigma^2} \ln  \frac{1}{ \mu^\ast_+ (t) - \mu}   $ (resp. $H (\mu) =  \frac{\omega^-_t}{ \mu^\ast_- (t) - \mu}  - \frac{2 a }{\sigma^2} \ln  \frac{1}{ \mu^\ast_- (t) - \mu} )   $.
\end{theorem}

\begin{corollary}{\label{asymptot_I}}  
Let $t > 0$ and $Z := \lambda_1 V_t + \lambda_2 I_t$, for $ \lambda_1, \lambda_2 \in \mathbb{R}$ such that $\max (\lambda_1, \lambda_2) > 0$. Then  we have 
 \begin{equation}\label{log}
 \limsup_{ R \to \infty}  \frac{ \ln \mathbb{P} ( Z  > R ) + \mu^\ast_+(t) R  - 2 \sqrt{\omega^+_t R }}{\ln (R)}  \in \left[  \frac{a}{\sigma^2} - \frac{3}{4}, \frac{a}{\sigma^2} \right]   
 \end{equation}
 and 
  \begin{equation}\label{log}
 \liminf_{ R \to \infty}  \frac{ \ln \mathbb{P} ( Z  > R ) + \mu^\ast_+(t) R  - 2 \sqrt{\omega^+_t R }}{\ln (R)}  \le   \frac{a}{\sigma^2} - \frac{3}{4}.
  \end{equation}
 \end{corollary} 
 
 \begin{remark}
A similar result can be obtained for the complementary cumulative distribution of the stock price in Heston model. The analysis relies on the explicit calculation of the moment generating function of $X$. Let's define the quantity
\begin{equation}
F_p (t) = \mathbb{E} \; e^{p X_t}, ~~ p \in \mathbb{R}.
\end{equation}
We have for all $ p \in \mathbb{R}$
\begin{eqnarray*}
F_p (t)  &=&  \mathbb{E} \; e^{ - \frac{p}{2} \int_0^t V_s d s + p \rho \int_0^t \sqrt{V_s} d W^1_s + p \sqrt{1 - \rho^2} \int_0^t \sqrt{V_s} d W^2_s } \nonumber\\ &=&
 \mathbb{E} \left\{  e^{  \frac{p^2 (1 - \rho^2) -p}{2} \int_0^t V_s d s + p \rho \int_0^t \sqrt{V_s} d W^1_s   }
 \left[  \mathbb{E} \; e^{   p \sqrt{1 - \rho^2} \int_0^t \sqrt{V_s} d W^2_s - \frac{p^2 (1 - \rho^2)}{2} \int_0^t V_s d s} \Big| (W^1_s)_{s \le t} \right]  \right\} \nonumber\\ &=&
 \mathbb{E} \left[  e^{ p \rho \int_0^t \sqrt{V_s} d W^1_s - \frac{p^2 \rho^2}{2} \int_0^t V_s d s } \; e^{  \frac{p^2 -p}{2} \int_0^t V_s d s      }  \right] \nonumber\\ &=& 
  \mathbb{E}^\mathbb{Q} \;    e^{  \frac{p^2 -p}{2} \int_0^t V_s d s      }   ,
\end{eqnarray*}
where we used the law of iterated conditional expectation and the fact that $V_s$ is measurable with respect to $W^1$. The last inequality is a consequence of Girsanov theorem, where under $ \mathbb{Q}$, the process $V$ satisfies the stochastic differential equation
\begin{equation}
d V_t = \left(  a - ( b - \rho \sigma p )V_t\right)d t + \sigma \sqrt{V_t} d W^{1,\mathbb{Q}}_t
\end{equation}
with $ \mathbb{Q}-$ Brownian motion $W^{1,\mathbb{Q}} $. We are then reduced to the calculation of the moment generating function of the time average of the CIR process $V$ under $\mathbb{Q}$.  It follows that $F_p (t) = e^{a \varphi_p (t) + v \psi_p (t) }$, where $ \varphi_p (t) = \int_0^t \psi_p (s) d s $ and $ \psi_p$ is given by theorem~\ref{moment_esponentiels}.  
\end{remark}

\section{Conclusions}
\noindent We have presented a Tauberian result that relates the right tail behavior of the cumulative distribution function to the behavior of the moment generating function, for a random variable, whose moment generating function is finite only on part of the real line. We have shown that if the function $ \varphi : x \longmapsto \Lambda (\mu^\ast - \frac{1}{x}) :=  \ln \left( \mathbb{E} \; e^{(\mu^\ast - \frac{1}{x}) Z} \right)$ is regularly varying,  then the Fenchel-Legendre transform of $\Lambda$ gives a log-equivalent (we say $f$ and $g$ are log-equivalent if $\limsup_{x \to \infty} \frac{f(x) - g(x)}{\log (x)} =0$) of the logarithm of the complementary cumulative distribution, while the most recent result of this kind gives only an equivalence of type ($\lim_{x \to \infty} \frac{f(x) - g(x)}{x} =0$). Our results apply to financial models that generate distributions satisfying Assumption~\ref{hyp_laregdev} (several examples can be found in \cite{benaimfriz08}). We applied our results to an arbitrary superposition of the CIR process and its time integral. It would be interesting to have a similar result to Theorem~\ref{tauberian_theorem}  for $\mathbb{R}^d$ random variables. Such a result allows more understanding of the joint distribution of CIR process and its time integral as well as the conditional distribution of the CIR with respect to its time average. This will be subject to further investigation.

\newpage

 \appendix

\section{Solution of (\ref{eq-psi})}
\noindent To solve \eqref{eq-psi}, we write
\[
 \psi^2_{\lambda_1, \lambda_2 }(t)-2\frac{b}{\sigma^2}\psi_{\lambda_1, \lambda_2 }(t)
      +2\frac{\lambda_2}{\sigma^2}  =\left( \psi_{\lambda_1, \lambda_2 }(t)-\frac{b}{\sigma^2}\right)^2
            +\frac{2  \lambda_2 \sigma^2 -b^2}{\sigma^4}.
\]
There are three situations :
\begin{description}
\item[Case $  2  \lambda_2 \sigma^2 < b^2 $ ]
     Setting 
    \[
    \bar{\psi}_{\lambda_1, \lambda_2 }(t)=\psi_{\lambda_1, \lambda_2 }(t)-\frac{b}{\sigma^2}\;\mbox{ 
    and }\;\alpha=\sqrt{b^2-2  \lambda_2 \sigma^2}/\sigma^2,
    \]
    the equation \eqref{eq-psi} becomes
    \[
    \bar{\psi}'_{\lambda_1, \lambda_2 }(t)=\frac{\sigma^2}{2}\left(\bar{\psi}^2_{\lambda_1, \lambda_2 }(t)-\alpha^2\right),
    \] 
    which gives
    \[
    \frac{1}{2\alpha}\left(\frac{1}{\bar{\psi}_{\lambda_1, \lambda_2 }(t)-\alpha}
            -\frac{1}{\bar{\psi}_{\lambda_1, \lambda_2 }(t)+\alpha}\right) \bar{\psi}'_{\lambda_1, \lambda_2 }(t)
            =\frac{\sigma^2}{2}.
    \]
    Hence
    \[
    \ln\left|\frac{\bar{\psi}_{\lambda_1, \lambda_2 }(t)-\alpha}{\bar{\psi}_{\lambda_1, \lambda_2 }(t)+\alpha}\right|=
                     \alpha\sigma^2 t+\ln\left|\frac{\lambda_1   - \frac{b}{\sigma^2}-\alpha}{\lambda_1    - \frac{b}{\sigma^2}+\alpha}\right|.
    \]
  It follows that
\[
 \left\{ \begin{array}{ll}
 \bar{\psi}_{\lambda_1, \lambda_2 }(t) = - \alpha \frac{  ( C e^{ \alpha \sigma^2 t} +1 )^2 }{ C^2 e^{2 \alpha \sigma^2 t} -1}, & \textrm{if}~~ |\lambda_1    - \frac{b}{\sigma^2}| > \alpha, \\ \\
  \bar{\psi}_{\lambda_1, \lambda_2 }(t) = - \alpha \frac{  ( C e^{ \alpha \sigma^2 t} -1 )^2 }{ C^2 e^{2 \alpha \sigma^2 t} -1}, & \textrm{if}~~ |\lambda_1    - \frac{b}{\sigma^2}| \le \alpha,
 \end{array} \right.
\]    
where $\displaystyle C=\left|\frac{\lambda_1  - \frac{b}{\sigma^2}-\alpha}{\lambda_1    - \frac{b}{\sigma^2}+\alpha}\right|$.     Thus
    \[
 \left\{ \begin{array}{ll}
 \psi_{\lambda_1, \lambda_2 }(t) =\frac{b}{\sigma^2} - \alpha \frac{    C e^{ \alpha \sigma^2 t} +1  }{ C  e^{  \alpha \sigma^2 t} - 1}, & \textrm{if}~~ |\lambda_1   - \frac{b}{\sigma^2}| > \alpha, \\ \\
  \psi_{\lambda_1, \lambda_2 }(t) =\frac{b}{\sigma^2} - \alpha \frac{    C e^{ \alpha \sigma^2 t} -1   }{ C  e^{  \alpha \sigma^2 t} +1}, & \textrm{if}~~ |\lambda_1   - \frac{b}{\sigma^2}| \le \alpha.
 \end{array} \right.
\]

\item[Case $   2  \lambda_2 \sigma^2=b^2  $ ] 
The equation \eqref{eq-psi} can be written as
 \[
    \bar{\psi}'_{\lambda_1, \lambda_2 }(t)=\frac{\sigma^2}{2}\bar{\psi}^2_{\lambda_1, \lambda_2 }(t),
    \] 
    which gives
    \begin{eqnarray*}
\frac{1}{\bar{\psi}_{\lambda_1, \lambda_2 }(0)}-
\frac{1}{\bar{\psi}_{\lambda_1, \lambda_2 }(t)}&=&\frac{\sigma^2t}{2}.
\end{eqnarray*}
Thus
\[
\psi_{\lambda_1, \lambda_2 }(t)=\frac{ \lambda_1    - \frac{b}{\sigma^2} }{ 1 -\frac{\sigma^2}{2} ( \lambda_1  - \frac{b}{\sigma^2} )t} + \frac{b}{\sigma^2}.
\]
\item[Case $  2 \lambda_2 \sigma^2>b^2  $ ]
We set $\beta=\sqrt{2  \lambda_2 \sigma^2-b^2}/\sigma^2$. The equation \eqref{eq-psi} can be written as
\[
 \bar{\psi}'_{\lambda_1, \lambda_2 }(t)=\frac{\sigma^2}{2}\left(\bar{\psi}^2_{\lambda_1, \lambda_2 }(t)+\beta^2\right).
 \]
 It follows that
 \[
 \frac{\bar{\psi}'_{\lambda_1, \lambda_2 }(t)}{\bar{\psi}^2_{\lambda_1, \lambda_2 }(t)+\beta^2}=\frac{\sigma^2}{2}.
 \]
 Integrating both sides of this equation we obtain 
 \[
\frac{1}{\beta} \arctan (\bar{\psi}_{\lambda_1, \lambda_2 }(t)/\beta)=\frac{\sigma^2t}{2}+Cst.
 \]
 Thus
 \[
{\psi}_{\lambda_1, \lambda_2 }(t)= \frac{b}{\sigma^2}+
       \beta \tan\left(\beta\frac{\sigma^2t}{2} + \arctan\left( \frac{\lambda_1  \sigma^2 -b}{\beta\sigma^2} \right) \right).
 \]
\end{description}

Finally, the solution of \eqref{eq-psi} is given by the following table \\ \\ 
\begin{tabular}{|c|c|}

  \hline
  $ \lambda_2  < \frac{ b^2}{ 2 \sigma^2}$  &$ 
 \left\{ \begin{array}{ll}
 \psi_{\lambda_1, \lambda_2 }(t) =\frac{b}{\sigma^2} - \alpha \frac{    C e^{ \alpha \sigma^2 t} +1  }{ C  e^{  \alpha \sigma^2 t} - 1},~~~~~ \textrm{if}~~~~ |\lambda_1   - \frac{b}{\sigma^2}| > \alpha, \\  
  \psi_{\lambda_1, \lambda_2 }(t) =\frac{b}{\sigma^2} - \alpha \frac{    C e^{ \alpha \sigma^2 t} -1   }{ C  e^{  \alpha \sigma^2 t} +1},~~~~~ \textrm{if}~~~~ |\lambda_1   - \frac{b}{\sigma^2}| \le \alpha, 
 \end{array}  \right.
    $
   \\ 
   & where $\alpha =  \sqrt{b^2-2  \lambda_2 \sigma^2}/\sigma^2$  and $C= \left|\frac{\lambda_1  - \frac{b}{\sigma^2}-\alpha}{\lambda_1    - \frac{b}{\sigma^2}+\alpha}\right| $. \\
   \hline
   $ \lambda_2  = \frac{ b^2}{ 2 \sigma^2}$  & $\psi_{\lambda_1, \lambda_2 }(t)=\frac{ \lambda_1    - \frac{b}{\sigma^2} }{ 1 -\frac{\sigma^2}{2} ( \lambda_1  - \frac{b}{\sigma^2} )t} + \frac{b}{\sigma^2}.$ \\
 \hline
  $ \lambda_2  > \frac{ b^2}{ 2 \sigma^2}$  & $\psi_{\lambda_1, \lambda_2 }(t)= \frac{b}{\sigma^2}+
       \beta \tan\left(\beta\frac{\sigma^2t}{2} + \arctan\left( \frac{\lambda_1  \sigma^2 -b}{\beta\sigma^2} \right) \right),$ \\
       & where $\beta=\sqrt{2  \lambda_2 \sigma^2-b^2}/\sigma^2$.\\
  \hline  
\end{tabular}\\ 

and	 $t^\ast_{\lambda_1, \lambda_2}$ is given explicitly in terms of $\lambda_1$ and $\lambda_2$ by \\ \\
\begin{tabular}{|c|c|}
  \hline
  $ \lambda_2  < \frac{ b^2}{ 2 \sigma^2}$  and $ \lambda_1 > \frac{b}{\sigma^2}+ \sqrt{b^2-2  \lambda_2 \sigma^2}/\sigma^2 $ &$ 
 t^\ast_{\lambda_1, \lambda_2} = \frac{1}{\sqrt{b^2-2  \lambda_2 \sigma^2}} \ln \left(\frac{ \lambda_1 - b/\sigma^2 + \sqrt{b^2-2  \lambda_2 \sigma^2}/\sigma^2 }{ \lambda_1 - b/\sigma^2 -\sqrt{b^2-2  \lambda_2 \sigma^2}/\sigma^2}\right).
    $
   \\
   \hline
    $ \lambda_2  < \frac{ b^2}{ 2 \sigma^2}$  and $ \lambda_1 \le \frac{b}{\sigma^2}+ \sqrt{b^2-2  \lambda_2 \sigma^2}/\sigma^2 $ &$ 
 t^\ast_{\lambda_1, \lambda_2} = + \infty$.
    \\
   \hline
   $ \lambda_2  = \frac{ b^2}{ 2 \sigma^2}$  & $ t^\ast_{\lambda_1, \lambda_2} = \frac{2}{\sigma^2  ( \lambda_1  - \frac{b}{\sigma^2} )}    .$ \\
 \hline
  $ \lambda_2  > \frac{ b^2}{ 2 \sigma^2}$  & $t^\ast_{\lambda_1, \lambda_2}= \frac{2}{\sqrt{2  \lambda_2 \sigma^2-b^2}} \left( \frac{\pi}{2} -\arctan\left( \frac{\lambda_1  \sigma^2 -b}{\sqrt{2  \lambda_2 \sigma^2-b^2}} \right) \right).
  $ \\
  \hline  
\end{tabular}\\ 
\\

\section{Proof of Theorem~\ref{equiv_laplace}}
 \noindent The critical moment $ \mu^\ast_+ (t)$ is defined as solution of
 \[
  t^\ast_{\lambda_1, \lambda_2} (\mu^\ast_+ (t)) = t.
 \]    
There are only two situations for $t^\ast_{\lambda_1, \lambda_2}$:
\begin{description}
\item[Case $ t^\ast_{\lambda_1, \lambda_2} (\mu^\ast_+) = \frac{1}{\sqrt{b^2-2  \lambda_2 \mu^\ast_+\sigma^2}} \ln \left(\frac{ \lambda_1\mu^\ast_+ - b/\sigma^2 + \sqrt{b^2-2  \lambda_2\mu^\ast_+ \sigma^2}/\sigma^2 }{ \lambda_1\mu^\ast_+ - b/\sigma^2 -\sqrt{b^2-2  \lambda_2 \mu^\ast_+\sigma^2}/\sigma^2}\right) $] In this case we have, for $   \epsilon  $ positive and sufficiently small,
 \[ 
 \psi_{\lambda_1(\mu^\ast_+ - \epsilon), \lambda_2 (\mu^\ast_+ - \epsilon) }(t) =\frac{b}{\sigma^2} - \alpha \frac{    C e^{ \alpha \sigma^2 t} +1  }{ C  e^{  \alpha \sigma^2 t} - 1},
    \]
     where $\alpha = \alpha (\epsilon)=  \sqrt{b^2-2  \lambda_2 (\mu^\ast_+ - \epsilon)\sigma^2}/\sigma^2$  and 
     \[
     C=C(\epsilon)=  \frac{\lambda_1(\mu^\ast_+ - \epsilon)  - \frac{b}{\sigma^2}-\alpha}{\lambda_1(\mu^\ast_+ - \epsilon)    - \frac{b}{\sigma^2}+\alpha}  .
     \]
      We obtain, by writing the Taylor expansion of $C e^{ \alpha \sigma^2 t} $ with respect to $\epsilon$, that $\psi_{\lambda_1(\mu^\ast_+ - \epsilon), \lambda_2 (\mu^\ast_+ - \epsilon) }(t)$ can be written as
    \[
    \psi_{\lambda_1(\mu^\ast_+ - \epsilon), \lambda_2 (\mu^\ast_+ - \epsilon) }(t) = \frac{b}{\sigma^2} + \alpha (0) \; \frac{ 2 - c_1 \epsilon+ \mathcal{O} (\epsilon^2) }{ c_1 \epsilon+ \mathcal{O} (\epsilon^2)} = \frac{2 \alpha (0)/ c_1 }{ \epsilon } + e_1 (\epsilon) ,
    \]
    where $e_1 (\epsilon)$ is bounded. It follows that, for $\epsilon$ sufficiently small, 
  \[
  \ln \left( \mathbb{E} \; e^{ (\mu^\ast_+ - \epsilon)\lambda_1 V_t + (\mu^\ast_+ - \epsilon) \lambda_2 \int_0^t{ V_u d u}} \right)-  \frac{2 v \alpha (0)/ c_1 }{ \epsilon } ~ \sim ~ a  \varphi_{\lambda_1(\mu^\ast_+ - \epsilon), \lambda_2 (\mu^\ast_+ - \epsilon) }(t)  .
  \]
  On the other hand, we have
\begin{equation}
\varphi_{\lambda_1(\mu^\ast_+ - \epsilon), \lambda_2 (\mu^\ast_+ - \epsilon) }(t) = \frac{b}{\sigma^2} t - \alpha(\epsilon) \left(  \ln \left(  C e^{ \frac{\alpha \sigma^2}{2} t} - e^{ - \frac{\alpha \sigma^2}{2} t} \right) - \ln \left(  C -1\right) \right).
\end{equation}  
 We obtain that $\varphi_{\lambda_1(\mu^\ast_+ - \epsilon), \lambda_2 (\mu^\ast_+ - \epsilon) }(t) ~ = ~ \frac{2   }{\sigma^2} \ln \frac{1}{ \epsilon}+ e_2 (\epsilon), $ where $e_2 (\epsilon)$ is bounded.

\item[Case  $t^\ast_{\lambda_1, \lambda_2}= \frac{2}{\sqrt{2  \lambda_2 \sigma^2-b^2}} \left( \frac{\pi}{2} -\arctan\left( \frac{\lambda_1  \sigma^2 -b}{\sqrt{2  \lambda_2 \sigma^2-b^2}} \right) \right)
  $] In this case, for $0 < \epsilon \ll 1 $, we have
\[
 \left\{\begin{array}{l}
   \psi_{\lambda_1(\mu^\ast_+ - \epsilon), \lambda_2 (\mu^\ast_+ - \epsilon) }(t)=  \frac{b}{\sigma^2}+
       \frac{\sqrt{2  \lambda_2 (\mu^\ast_+ - \epsilon) \sigma^2-b^2}}{\sigma^2} \tan\left( g(t, (\mu^\ast_+ - \epsilon)) \right) ,\\ \\
     \varphi_{\lambda_1(\mu^\ast_+ - \epsilon), \lambda_2 (\mu^\ast_+ - \epsilon) }(t) =  \frac{b}{\sigma^2}  t+
       \frac{2}{\sigma^2}  \left(  \ln \cos  g(0, \mu^\ast_+ - \epsilon) - \ln \cos  g(t, \mu^\ast_+ - \epsilon)  \right),
         \end{array}
           \right.
 \]
 where
 \begin{equation}
       g (t, \mu) = \frac{\sqrt{2\mu \lambda_2 \sigma^2-b^2}}{ 2}  t + \arctan( \frac{ \lambda_1 \mu \sigma^2 -b}{ \sqrt{2\mu \lambda_2 \sigma^2-b^2}}).
 \end{equation}
We show by the same as before that for $ \epsilon$ sufficiently small, 
\[
\ln \left( \mathbb{E} \; e^{ (\mu^\ast_+ - \epsilon) \lambda_1 V_t + (\mu^\ast_+ - \epsilon) \lambda_2 \int_0^t{ V_u d u}} \right)  = \frac{x \sqrt{2 \mu^\ast_+ \sigma^2-b^2}}{\sigma^2 \partial_\mu g (t, \mu^\ast_+) }   \frac{  1 }{   \epsilon }   - \frac{2 a }{\sigma^2} \ln ( \epsilon) + e(\epsilon) ,
\] 
 where $e (\epsilon)$ is bounded.
\end{description}

\section*{Acknowledgments}
This research has been supported by a grant from Riksbankens Jubileumsfond (P10-0113:1). The author would also like to thank Professor Damien Lamberton for many useful discussions and anonymous referee who read the first version and helped me improve the presentation.

\end{document}